\documentclass[letterpaper,11pt]{amsart}

\usepackage[margin=1.4in]{geometry}
\usepackage{amsmath}
\usepackage{amsfonts}
\usepackage{amscd, amssymb}
\usepackage{amsmath,amscd}
\usepackage{amsthm}
\usepackage{mathtools}
\usepackage{hyperref}
\usepackage{graphics}
\usepackage{graphicx}
\usepackage{color}
\usepackage{bm}
\usepackage{enumitem}
\setlist[enumerate,1]{label={(\alph*)}}
\setlist[enumerate,2]{label={(\roman*)}}

\newtheorem{thm}{Theorem}[section]

\newtheorem{prop}[thm]{Proposition}

\newtheorem{lem}[thm]{Lemma}

\newtheorem{claim}[thm]{Claim}

\newtheorem{observation}[thm]{Observation}
\newtheorem{cor}[thm]{Corollary}
\newtheorem{clm}[thm]{Claim}

\theoremstyle{definition}
\newtheorem{definition}[thm]{Definition}

\newtheorem*{notation*}{Notation}

\theoremstyle{remark}
\newtheorem{rmk}[thm]{Remark}

\newtheorem{algorithm}[thm]{Algorithm}

\newcommand{\ignore}[1]{}
\newcommand{\R}{\mathbb R}

\newcommand{\N}{\mathbb N}
\newcommand{\mA}{\mathcal A}
\newcommand{\mB}{\mathcal B}

\newcommand{\E}{{\mathbb{E}}}

\newcommand{\oone}{{o \left(1\right)}}
\newcommand{\omegaone}{{\omega \left(1\right)}}
\newcommand{\eps}{{\varepsilon}}
\newcommand{\termdefine}[1]{\textbf{#1}}
\newcommand{\thmref}[1]{{Theorem \ref{#1}}}

\newcommand{\lemref}[1]{{Lemma \ref{#1}}}
\newcommand{\claref}[1]{{Claim \ref{#1}}}
\newcommand{\Prob}{{\mathbb{P}}}
\newcommand{\prob}{{\Prob}}
\newcommand{\rt}{\right}
\newcommand{\lt}{\left}
\newcommand{\gnp}{{G\left(n;p\right)}}

\newcommand{\tg}{{\tilde{G}}}
\newcommand{\tgone}{{\tilde{G_1}}}
\newcommand{\tGH}{{\tg_H}}
\newcommand{\tS}{{\tilde{S}}}
\newcommand{\tT}{{\tilde{T}}}
\newcommand{\tDelta}{{\tilde{\Delta}}}

\newcommand{\dcup}{{\mathbin{\dot{\cup}}}}

\newcommand{\ldc}{{\frac{1}{1000}}}
\newcommand{\econst}{{\frac{1}{2000}}}
\newcommand{\egconst}{{\frac{1}{3000}}}

\DeclareMathOperator{\deg1}{deg}
\DeclareMathOperator{\var}{Var}

\begin{document}
\title{Perfect Matchings in Random Subgraphs of Regular Bipartite Graphs}
\author{Roman Glebov}\thanks{Roman Glebov was supported by ERC grant 678765 and ISF grant 1452/15.}
\address{Hebrew University of Jerusalem, Jerusalem 91904, Israel}
\email{roman.l.glebov@gmail.com}
\author{Zur Luria}
\address{Software Department, Jerusalem College of Engineering}
\email{zluria@gmail.com}
\author{Michael Simkin}
\address{Institute of Mathematics and Federmann Center for the Study of Rationality, The Hebrew University of Jerusalem, Jerusalem 91904, Israel}
\email{menahem.simkin@mail.huji.ac.il}

\begin{abstract}

Consider the random process in which the edges of a graph $G$ are added one by one in a random order. A classical result states that if $G$ is the complete graph $K_{2n}$ or the complete bipartite graph $K_{n,n}$, then typically a perfect matching appears at the moment at which the last isolated vertex disappears. We extend this result to arbitrary $k$-regular bipartite graphs $G$ on $2n$ vertices for all $k = \omega \left( \frac{n}{\log^{1/3} n} \right)$. 

Surprisingly, this is not the case for smaller values of $k$. Using a construction due to Goel, Kapralov and Khanna, we show that there exist bipartite $k$-regular graphs in which the last isolated vertex disappears long before a perfect matching appears.

\end{abstract}

\maketitle

\section{Introduction}\label{sec:intro}

The study of the random graph model $\gnp$ began with two influential papers by Erd\H{o}s and R\'enyi \cite{erdds1959random, erds1960evolution}. In \cite{erdds1959random} and \cite{erdos1964random}, they considered the range $p = \Theta(\log n /n)$ and the appearance of spanning structures in that regime. Later, several papers \cite{ajtai1981longest, bollobas1984evolution, komlos1973hamilton, komlos1983limit, korshunov1976solution, posa1976hamiltonian} led to the following understanding. Consider a random graph process on $n$ vertices, in which edges are added one by one in a random order. Asymptotically almost surely\footnote{An event occurs ``asymptotically almost surely'' (a.a.s.) if the probability of its occurrence tends to $1$ as $n\to\infty$. We say that a property holds for ``almost every'' element of a set if it holds a.a.s.\ for a uniformly random element of the set.}, the first edge that makes the minimum degree one connects the graph, and creates a perfect matching. Likewise, when the minimum degree becomes two, the graph immediately contains a Hamilton cycle. Philosophically, spanning structures appear once local obstructions disappear.

For a graph $G = (V,E)$ and $p \in [0,1]$, let $G(p)$ denote the distribution on subgraphs of $G$ in which each edge is retained with probability $p$, independently of the other edges. Recently, a series of papers \cite{krivelevich2014robust, glebov2017threshold, krivelevich2015long, riordan2014long} extended the above philosophy to $G(p)$ for various $G$. For example, in \cite{krivelevich2014robust} it was shown that if $G$ is a Dirac graph, then the threshold for Hamiltonicity of $G(p)$ remains $\Theta \left( \log n / n \right)$. See \cite{sudakov2017robustness} for a survey of these and related results.

In this paper we consider the threshold $p_0$ for the appearance of a perfect matching in $G(p)$ where $G$ is a $k$-regular bipartite graph on $2n$ vertices. The celebrated permanent inequalities of Bregman \cite{bregman1973some} and Egorychev--Falikman \cite{egorychev1981solution, falikman1981proof} imply that the number of perfect matchings in $G$ is $\left((1+o(1))\frac{k}{e}\right)^n$. In particular, this number depends little on the specific structure of $G$. It is therefore natural to conjecture that $p_0$ depends only on $n$ and $k$. Furthermore, the logical candidate is the threshold for the disappearance of isolated vertices in $G(p)$, which is $p = \Theta(\log n / k)$.

Indeed, Goel, Kapralov, and Khanna \cite[Theorem 2.1]{goel2010perfect} showed that there exists a constant $c$ such that for any $k\leq n$, if $p = c n \log n / k^2$, then with high probability $G(p)$ contains a perfect matching. In particular, if $k = \Omega(n)$, $p = O(\log n/k)$ suffices.

For $k = \omega \left( \frac{n}{\log^{1/3} n} \right)$ we considerably strengthen this result. Namely, we show that if one reconstructs $G$ by adding its edges one by one in a random order, then typically a perfect matching appears at the same moment that the last isolated vertex vanishes. As a consequence, it follows that for any $C>1$, if $p = C \log(n)/k$, then with high probability $G(p)$ contains a perfect matching.

Formally, a \termdefine{graph process in $G = (V,E)$} is a sequence of graphs
\[
(V,\emptyset) = G_0, G_1, \ldots, G_{|E|} = G
\]
on the vertex set $V$, where for each $i$, $G_i$ is obtained from $G_{i-1}$ by adding a single edge of $G$. The \termdefine{hitting time} of a monotone graph property $P$ with respect to a graph process is $\min\{t: G_t \in P\}$. 

For a graph process $\tilde{G}$, let $\tau_M(\tilde{G})$ and $\tau_I(\tilde{G})$ denote the hitting times for containing a perfect matching and having no isolated vertices, respectively. Clearly, for every graph process $\tilde{G}$ we have $\tau_M(\tilde{G})\geq \tau_I(\tilde{G})$. Our main result is that if $G$ is sufficiently dense and $\tilde{G}$ is chosen uniformly at random, equality a.a.s.\ holds.

\begin{thm}\label{thm:main}
	Let $k = \omega \left( \frac{n}{\log^{1/3} n} \right)$, let $G$ be a $k$-regular bipartite graph on $2n$ vertices, and let $\tilde{G}$ be a uniformly random graph process in $G$. Then, a.a.s.\  $\tau_M(\tilde{G})=\tau_I(\tilde{G})$.
\end{thm}

\begin{cor}
   For $G$ and $k$ as above, 
   \begin{itemize}
	   \item If $p = \frac{\log n - \omega(1)}{k}$, then a.a.s.\ $G(p)$ does not contain a perfect matching.
	   \item If $p = \frac{\log n + \omega(1)}{k}$, then a.a.s.\ $G(p)$ contains a perfect matching.
   \end{itemize}
\end{cor}

Quite surprisingly, it turns out that these results fail when $k$ is significantly smaller than $n / \log^{1/3} n$. We analyze a construction of Goel, Kapralov, and Khanna \cite{goel2010perfect} in which the threshold for a perfect matching is much larger than the threshold for the disappearance of isolated vertices.

\begin{prop}\label{prop:counterexample}
	There exist infinitely many $k$-regular bipartite graphs $G$ on $n$ vertices, with $k=\Omega \lt(\frac{n}{\log(n) \cdot \log(\log(n))}\rt)$, such that a.a.s.\ the random subgraph $G(p)$ does not contain a perfect matching for any $p \leq 2 \log n / k$. On the other hand, if $p = \left( \log n + \omegaone \right) / k$, then a.a.s.\ $G(p)$ contains no isolated vertices.
\end{prop}

We prove Proposition \ref{prop:counterexample} in Appendix \ref{app:proof_counterexample}.

Theorem \ref{thm:main} is almost a triviality if one assumes that $G$ is pseudorandom (cf.\ \cite[Lemma 3.1]{luria2019threshold}). The main element needed in our proof is a way to control induced subgraphs of $G$ with high discrepancy. To this end we prove a result on the structure of high discrepancy sets in sufficiently dense, regular, bipartite graphs (Lemma \ref{lem:structure}).

The remainder of this paper is organized as follows. Section \ref{ssec:notation} introduces our notation. In Section \ref{sec:structural} we prove Lemma \ref{lem:structure}, and in Section \ref{sec:pseudorandom} we establish some probabilistic tools. Finally, in Section \ref{sec:proof}, we prove Theorem \ref{thm:main}.

\subsection{Notation}\label{ssec:notation}

Throughout the paper, we disregard floor and ceiling signs to improve readability. Large real numbers should be rounded to the nearest integer. We denote by ``$\log$'' the natural logarithm.

For an integer $m\in\N$, we define $[m] = \left\{ 1,2,\ldots,m \right\}$. Let $X$ be a set and let $f: [|X|] \to \R$. We sometimes abuse notation by writing
\[
\sum_{|S|=1}^{m}\binom{|X|}{|S|}f(|S|)
=
\sum_{S \subseteq X : |S| \in [m]} f(|S|).
\]

Let $f,g: \N \to \R$. We write $f= \tilde{O} (g)$ if, for some $c>0$ and all large enough $n \in \N$, $f(n) \leq g(n) \log^c \left( g(n) \right)$.

Let $G = (V,E)$ be a graph. For $A,B \subseteq V$, denote by $E_G(A,B)$ the set of edges incident to both $A$ and $B$, and let $e_G(A,B) = |E_G(A,B)|$. Let $N_G(A)$ denote the set of \termdefine{neighbors} of $A$, i.e., the set $\left\{ v \in V : \exists a \in A \text{ s.t.\ } av \in E \right\} \setminus A$. We define $G \setminus A$ to be the induced graph on the vertex set $V(G)\setminus A$.

Suppose $G$ is a bipartite graph with vertex partition $X,Y$. A vertex set $A$ is \termdefine{partite} if $A \subseteq X$ or $A \subseteq Y$. We denote by $A^c$ the complement of $A$ w.r.t.\ its own part, i.e., $X \setminus A$ if $A \subseteq X$ and $Y \setminus A$ if $A \subseteq Y$. If $A$ is empty, it will be clear from context whether $A^c = X$ or $A^c = Y$.

By a common abuse of notation, we speak of $G(p)$ as having a certain property, instead of saying that $G \sim G(p)$ has that property.

In certain places we will need to show that events not only occur a.a.s., but that the probability of their non-occurence decays at a polynomial rate. We will say that such events occur \termdefine{with very high probability} (\termdefine{w.v.h.p.}). Formally, we say that a sequence of events $\{A_n\}_{n\in\N}$ occurs w.v.h.p.\ if $\log \left(\Prob [A_n^c] \right) = - \Omega \left( \log n \right)$.

\section{A Structural Lemma}\label{sec:structural}

Throughout this section $G = (X \dcup Y, E)$ is a $k$-regular bipartite graph on $2n$ vertices. A \termdefine{cut} in $G$ is a pair $(S,T)$ where $S \subseteq X$ and $T \subseteq Y$. We call $(S,T)$ a \termdefine{Hall cut} if $|S| > |T|$ and $N(S) \subseteq T$.
Hall's marriage theorem states that a balanced bipartite graph contains a perfect matching if and only if it contains no Hall cuts. The main idea in the proof of Theorem \ref{thm:main} is to show that a.a.s.\ $G_{\tau_I}$ does not contain a Hall cut.

Let $(S,T)$ be a cut in $G$. We call $E_G(S,T^c)$ the \termdefine{outgoing edges} of $(S,T)$. The \termdefine{cross edges} of $G$ with respect to $(S,T)$ are those in $E(S \cup T, S^c \cup T^c)$. We call the remaining edges \termdefine{parallel}. For a vertex $x \in V(G)$, we denote by $\deg1_{G,S,T}^{\mathrm{Par}}(x)$ and $\deg1_{G,S,T}^{\mathrm{Cr}}(x)$ the number of parallel and cross edges incident to $x$, respectively. Similarly, we denote by $N_{G,S,T}^{\mathrm{Par}}(x)$ the set of neighbors of $x$ that are connected to $x$ by a parallel edge. If the cut $(S,T)$ is clear from the context, we sometimes write $\deg1_G^{\mathrm{Par}} (x)$ and $\deg1_G^{\mathrm{Cr}}(x)$.

We define the following distance function on the set of cuts in $G$:
\[
d((S_1,T_1),(S_2,T_2)) = |S_1 \setminus S_2|+ |S_2 \setminus S_1|+ |T_1 \setminus T_2|+ |T_2 \setminus T_1|.
\]
For $C \in \mathbb{R}$, we say that two cuts are \termdefine{$C$-close} if their distance is at most $C$.

\begin{observation}\label{obs:crossedges}
	Let $(S,T)$ be a cut in $G$. Then $e(S,T^c)=k\cdot (|S|-|T|) + e(S^c, T)$.
\end{observation}

\begin{proof}
	Since $G$ is $k$-regular we have:
	\begin{align*}
	& e(S,T)+e(S,T^c) = k \cdot |S| \\
	& e(S,T)+e(S^c,T)=k \cdot |T|.
	\end{align*}
	Subtracting the second equation from the first yields the result.
\end{proof}

\begin{observation}\label{obs:probHall}
	Let $(S,T)$ be a cut in $G$ with $|S|>|T|$. Let $C = e_G(S\cup T, {S^c\cup T^c})$ be the number of cross edges in $G$ w.r.t.\ $(S,T)$. Then, for any $p \in (0,1)$, it holds that:
	\[
	\prob \lt[(S,T)\mbox{ is a Hall cut in } G(p) \rt] \leq (1-p)^{C/2}.
	\]
\end{observation}

\begin{proof}
	We have $e_G(S\cup T, {S^c\cup T^c}) = e_G(S,T^c) + e_G(S^c,T)$. As $|S|>|T|$, Observation \ref{obs:crossedges} implies that $e_G(S,T^c) > e_G(S^c,T)$ and therefore $e_G(S,T^c) > e_G(S\cup T, {S^c\cup T^c})/2$. The probability that none of these cross edges are edges in $G(p)$ (and thus $(S,T)$ is a Hall cut) is therefore bounded from above by $(1-p)^{C/2}$, as desired.
\end{proof}

The following structural lemma is the heart of our proof. Observation \ref{obs:probHall} implies that if $G$ has almost no cuts with few cross edges, a union bound is enough to show that a.a.s.\ $G(p)$ contains no Hall cuts. This is the case, for example, in random regular graphs. However, in an arbitrary graph this need not hold. Therefore, we must understand the behavior of cuts with few cross edges, and hence a significant chance of being Hall cuts in $G(p)$. We show that in any sufficiently dense, regular, bipartite graph, all such cuts can be grouped into a small (specifically, subpolynomial) number of equivalence classes. This allows us to control the contribution of these cuts to the probability that $G(p)$ contains a Hall cut.

\begin{definition}\label{def:internal}
	Let $c > 0$. A cut $(S,T)$ is \termdefine{$c$-internal} if it has at most $4 c nk / \log n$ cross edges.
\end{definition}
If a cut is $1$-internal, we sometimes just say that it is \termdefine{internal}. Note that $(S,T)$ is $c$-internal if and only if its complement $(S^c,T^c)$ is $c$-internal. Indeed, both cuts have the same cross edges.

\begin{lem}\label{lem:structure}
	Let $G = (X \dot{\cup} Y, E)$ be a $k$-regular bipartite graph on $2n$ vertices, with $k = \omega \left( \frac{n}{\log^{1/3} n} \right)$ and $n$ sufficiently large. Set $\varepsilon = \frac{n}{k \log^{1/3} (n)} = \oone$. There exist $m = 2^{\Theta(n / k)}$ and cuts $(S_1,T_1),\ldots,(S_m,T_m)$ with the following properties.
	\begin{enumerate}
		\item For every $i \in [m]$ and $x \in V(G)$, we have $\deg1_{G,S_i,T_i}^{\mathrm{Cr}}(x) \leq (1+\varepsilon)\frac{k}{2}$.
		\item Every internal cut $(S,T)$ with $|S| > |T|$ is $\varepsilon k$-close to $(S_i,T_i)$ for some $i \in [m]$.
	\end{enumerate}
\end{lem}

\begin{rmk}
	For graphs satisfying $k = \Omega (n)$ it is relatively straightforward to derive Lemma \ref{lem:structure} from Szemer\'edi's regularity lemma (albeit with a vastly larger bound on $m$). Alternatively, one could use the decomposition of dense regular graphs into ``robust components'' (induced subgraphs with good expansion properties) due to K{\"u}hn, Lo, Osthus, and Staden \cite[Theorem 3.1]{kuhn2014robust}.
\end{rmk}

In order to prove lemma \ref{lem:structure}, we first show that the internal cuts have a lattice-like structure. The following claim implies that for sufficiently small $c$, the distance between two $c$-internal cuts is either very large or very small.

\begin{claim}\label{clm:nearorfar}
	Let $(S_1,T_1)$ and $(S_2,T_2)$ be two $c$-internal cuts. Then 
	\[
	d((S_1,T_1),(S_2,T_2))\leq \frac{40 c n}{\log n} \text{ or } d((S_1,T_1),(S_2,T_2)) \geq \frac{k}{10}.
	\]
\end{claim}

\begin{proof}
	Let $d = d((S_1,T_1),(S_2,T_2))$. Without loss of generality assume that $|S_2 \setminus S_1| \geq d/4$. As $(S_1,T_1)$ is $c$-internal, we have:
	\[
	e_G \lt( S_2 \setminus S_1, T_1 \rt) \leq e_G \left( S_1^c,T_1 \right) \leq \frac{4cnk}{\log n}.
	\]
	Similarly, since $(S_2,T_2)$ is $c$-internal, we have:
	\[
	e_G \left( S_2 \setminus S_1, T_2 \right) =  e_G \left( S_2 \setminus S_1, Y \right) - e_G \left( S_2 \setminus S_1, T_2^c \right) \geq e_G \left( S_2 \setminus S_1, Y \right) - \frac{4cnk}{\log n}.
	\]
	Since $G$ is $k$-regular, we have:
	\[
	e_G \left( S_2 \setminus S_1, Y \right) = k \left| S_2 \setminus S_1 \right| \geq \frac{kd}{4}.
	\]
	Therefore:
	\begin{equation}\label{eq:upper}
	e_G \left( S_2 \setminus S_1, T_2 \setminus T_1 \right) \geq e_G \left( S_2 \setminus S_1 , T_2 \right) - e_G \left( S_2 \setminus S_1 , T_1 \right) \geq \frac{kd}{4} - \frac{8cnk}{\log n}.
	\end{equation}
	On the other hand, it is certainly true that $e_G \left( S_2 \setminus S_1, T_2 \setminus T_1 \right) \leq \left| S_2 \setminus S_1 \right| \left| T_2 \setminus T_1 \right|$. As $\left| S_2 \setminus S_1 \right| + \left| T_2 \setminus T_1 \right| \leq d$, we have:
	\begin{equation}\label{eq:lower}
	e_G \left( S_2 \setminus S_1, T_2 \setminus T_1 \right) \leq \frac{d^2}{4}.
	\end{equation}
	Combining \eqref{eq:upper} and \eqref{eq:lower} and rearranging yields:
	\[
	d \left( k - d \right) \leq \frac{32 c nk}{\log n}.
	\]
	Suppose $d < k/10$. Then $k-d > 9k/10$. We thus obtain the inequality:
	\[
	d \leq \frac{320 c n}{9 \log n} < \frac{40 c n}{\log n},
	\]
	as desired.
\end{proof}

\begin{rmk}
	Although Claim \ref{clm:nearorfar} holds for all $c$, it is only meaningful when $40 c n/\log n < k/10$. For convenience, let $\delta = k/n$ denote the density of $G$. We will need to apply the claim for $c = O(1/\delta^2)$. Since $\delta = \omega \left( \log^{-1/3} (n) \right)$, this is in the regime where the claim is meaningful. In fact, this is the source of the lower bound on $k$ in Theorem \ref{thm:main}.
	
	For the rest of this section, $c$ will always be bounded by $O(1/\delta^2)$.
\end{rmk}

We say that two $c$-internal cuts are \termdefine{equivalent} if they are $\left( \eps k/100 \right)$-close. The triangle inequality, together with Claim \ref{clm:nearorfar}, implies that this is an equivalence relation. Let $\mathcal{X}_c$ be the set of equivalence classes of $c$-internal cuts. We say that a cut is \termdefine{trivial} if it is equivalent to $(\emptyset,\emptyset)$.

Note that the quantity $\varepsilon k$ depends only on $n$, and not on $k$. We choose to write $\varepsilon k$ in order to emphasize that this is asymptotically smaller than $k$.

We now define a meet operation on equivalence classes. Note that if the cuts $(S_1,T_1)$ and $(S_2,T_2)$ are $c_1$-internal and $c_2$-internal, respectively, then the cut $\left( S_1 \cap S_2 , T_1 \cap T_2 \right)$ is $(c_1+c_2)$-internal. Indeed, any cross edge of the intersection is a cross edge in at least one of the cuts.

Denote the equivalence class of a $c$-internal cut $(S,T)$ by $\left[ (S,T) \right]_c$. When the value of $c$ is clear from the context, we omit the subscript.

\begin{definition} 
	Let $c_1$ and $c_2$ satisfy $c_1+c_2 = O(1/\delta^2)$. The \termdefine{meet} of $[(S_1,T_1)] \in \mathcal{X}_{c_1}$ and  $[(S_2,T_2)] \in \mathcal{X}_{c_2}$ is 
	\[
	[(S_1,T_1)] \wedge [(S_2,T_2)] \coloneqq [(S_1\cap S_2,T_1\cap T_2)] \in \mathcal{X}_{c_1+c_2}.
	\]
\end{definition}

The fact that this is well-defined, in the sense that it does not depend on the choice of representatives, follows from Claim \ref{clm:nearorfar}. Indeed, different choices of representatives may only change the intersection by at most $80(c_1+c_2)n/\log n$ vertices. This is smaller than $\eps k/100$, and therefore the two intersections are equivalent. 

\begin{definition}
	We say that the class $[(S_1,T_1)]\in \mathcal{X}_{c_1}$ is \termdefine{above} $[(S_2,T_2)] \in \mathcal{X}_{c_2}$ if $[(S_1^c,T_1^c)] \wedge [(S_2,T_2)]$ is trivial. In this case, we also say that $[(S_2,T_2)]$ is \termdefine{below} $[(S_1,T_1)]$.
\end{definition}

Note that if $[(S,T)]$ is a nontrivial class then by Claim \ref{clm:nearorfar} $|S \cup T| \geq k/10$.

\begin{observation}\label{obs:large_difference}
	If $[(S_1,T_1)]$ is not below $\lt[(S_2,T_2)\rt]$, then for every $(S',T') \in [(S_1,T_1)] \wedge [(S_2,T_2)]$, we have $|S_1 \cup T_1| - |S'\cup T'| \geq k/20$.
\end{observation}

\begin{proof}
	Since $\lt[(S_1,T_1)\rt]$ is not below $[(S_2,T_2)]$, it holds that $[(S_1\setminus S_2, T_1\setminus T_2)]$ is nontrivial, and therefore $|S_1\setminus S_2|+|T_1\setminus T_2| \geq k/10$. Note that $S_1$ is the disjoint union of $S_1\cap S_2$ and $S_1 \setminus S_2$, and that a similar statement holds for $T_1$. This implies that $|S_1\cup T_1| \geq |S_1 \cap S_2| + |T_1 \cap T_2| + k/10$. As $(S',T')$ is equivalent to $(S_1\cap S_2,T_1\cap T_2)$, the observation follows.
\end{proof}

We now construct the building blocks used to create the cuts of Lemma \ref{lem:structure}.
Consider the following process:
\begin{algorithm}\label{alg:atom constructor}\hfill
\begin{enumerate}
	\item Initialize $[(S_1,T_1)]\in \mathcal{X}_1$ to be an arbitrary nontrivial internal equivalence class.
	
	\item As long as there exists a class $[(S^*,T^*)] \in \mathcal{X}_1$ that is neither above, nor has trivial meet with, $[(S_i,T_i)]$, set $[(S_{i+1},T_{i+1})]=[(S_i,T_i)] \wedge [(S^*,T^*)]$.
\end{enumerate}
\end{algorithm}

Note that this process halts after at most $40 / \delta$ steps. This is because $|S_1\cup T_1| \leq 2n$, and Observation \ref{obs:large_difference} implies that for each $i$, $|S_i \cup T_i| \leq \left| S_{i-1} \cup T_{i-1} \right| - k/20$. In consequence, the equivalence classes obtained at the end of this process are $(40/\delta)$-internal, because the meet of a $c$-internal cut with a $1$-internal cut is $(c+1)$-internal. We may therefore think of these equivalence classes as members of $\mathcal{X}_{40/\delta}$. We call these classes \termdefine{atoms}. 

\begin{observation}\label{obs:atom properties}\leavevmode

	\begin{enumerate}
		
		\item\label{itm:nontrivial} All of the atoms are nontrivial.
		
		\item\label{itm:disjoint} Every pair of atoms has trivial meet.
		
		\item\label{itm:atom count} There are at most $30 / \delta$ atoms.
	\end{enumerate}
\end{observation}

\begin{proof}
	Item \ref{itm:nontrivial} holds because at each stage of the process, $[(S_i,T_i)]$ and $[(S^*,T^*)]$ have nontrivial meet.
	
	For Item \ref{itm:disjoint}, assume that $[(S,T)]$ and $[(S',T')]$ are distinct atoms. By definition of the equivalence relation, $d((S,T),(S',T')) \geq \eps k/100$. By Claim \ref{clm:nearorfar} this implies that in fact $d((S,T),(S',T')) \geq k/10$. Without loss of generality, we may assume that ${|S\setminus S'|} + {|T \setminus T'|} \geq k / 20$. Now, $[(S',T')]$ is the meet of at most $40/\delta$ 1-internal classes. Thus $(S',T')$ is $\varepsilon k/100$-close to the intersection of at most $40/\delta$ $1$-internal cuts. Therefore at least one of these classes, $[(S^*,T^*)]$, satisfies 
	\[
	|S \setminus S^*|+|T \setminus T^*| \geq \left( \frac{k}{20} + \frac{\varepsilon k}{100} \right) \cdot \frac{\delta}{40}.
	\]
	Since $k = \omega \left( \frac{n}{\log^{1/3} n} \right)$, this is larger than $\frac{40}{\delta} \cdot \frac{40 n}{\log n}$. Therefore, by Claim \ref{clm:nearorfar}, this implies that $[(S,T)]\wedge [(S^{*c},T^{*c})]$ is nontrivial, and therefore $[(S^*,T^*)]$ is not above $[(S,T)]$. Suppose, for a contradiction, that $[(S,T)]$ and $[(S',T')]$ have nontrivial meet. Then $[(S,T)]$ and $[(S^*,T^*)]$ also have nontrivial meet. Therefore $[(S,T)]$, by definition, is not an atom, which is a contradiction.
	
	For Item \ref{itm:atom count}, fix a representative for each atom. Each representative contains at least $k/10$ vertices. Since each pair of atoms has trivial meet, by Claim \ref{clm:nearorfar}, the intersection of any two representatives has at most $\frac{1600n^2}{k \log n}$ vertices. Letting $A$ denote the number of atoms, the inclusion-exclusion formula implies that for all $a \leq A$:
	\[
	a \cdot \lt( \frac{k}{10}\rt) - \binom{a}{2} \frac{1600n^2}{k \log n} \leq |V(G)| = 2n.
	\]
	The inequality does not hold for $a = 30/\delta$, and therefore $A \leq 30/\delta$.
\end{proof}

\begin{claim}\label{clm:atom internal intersection}
	Suppose $[(S,T)]$ is internal and above a nontrivial $O(1/\delta^2)$-internal class $[(S',T')]$. Then there exists an atom $[(S_A,T_A)]$ below $[(S,T)]$ such that $[(S_A,T_A)] \wedge [(S',T')]$ is nontrivial.
\end{claim}

\begin{proof}
	We will construct $[(S_A,T_A)]$ using Algorithm \ref{alg:atom constructor}, while ensuring throughout that the meet $[(S_i,T_i)] \wedge [(S',T')]$ is nontrivial. Here, $[(S_i,T_i)]$ refers to the sequence of cuts constructed in Algorithm \ref{alg:atom constructor}.
	
	Set $[(S_1,T_1)] = [(S,T)]$. Suppose that at step $i \leq 40/\delta$ we have a class $[(S_i,T_i)]$ such that $[(S_i,T_i)] \wedge [(S',T')]$ is nontrivial. If $[(S_i,T_i)]_{40/\delta}$ is an atom take $[(S_A,T_A)] = [(S_i,T_i)]_{40/\delta}$. Otherwise, there exists an internal cut $[(S^*,T^*)]$ that is neither above, nor has trivial meet with, $[(S_i,T_i)]$. Therefore both $[(S_i,T_i)] \wedge [(S^*,T^*)]$ and $[(S_i,T_i)] \wedge [(S^*,T^*)]^c \coloneqq [((S^*)^c,(T^*)^c)]$ are nontrivial.
	
	Additionally, since $[(S_i,T_i)] \wedge [(S',T')]$ is nontrivial, it holds that at least one of $[(S^*,T^*)] \wedge [(S_i,T_i)] \wedge [(S',T')]$ and $[(S^*,T^*)]^c \wedge [(S_i,T_i)] \wedge [(S',T')]$ are nontrivial. Without loss of generality assume that $[(S^*,T^*)] \wedge [(S_i,T_i)] \wedge [(S',T')]$ is nontrivial and set $[(S_{i+1},T_{i+1})] = [(S^*,T^*)] \wedge [(S_i,T_i)]$. Observe that $[(S_{i+1},T_{i+1})] \wedge [(S',T')]$ is nontrivial and $O(1/\delta^2)$-internal.
	
	Note that although it may look like the distance between $(S_i,T_i)$ and the trivial cut is halved at each step, in fact Claim \ref{clm:nearorfar} implies that the distance is never less than $k/10$. This implicitly uses the fact that $[(S_i,T_i)] \wedge [(S',T')]$ is $O\left(1/\delta^2\right)$-internal.
\end{proof}

\begin{proof}[Proof of Lemma \ref{lem:structure}]
	\newcommand{\mS}{{\mathcal{S}}}
	
	We first construct the cuts $(S_1,T_1),\ldots,(S_m,T_m)$. Let $\mA$ be the set of atoms. For each $\alpha \in \mA$, fix a representative $\left( S_\alpha,T_\alpha \right)$. For $\mS \subseteq \mA$, let $\left(S'_\mS,T'_\mS \right)$ be the cut $\left( \cup_{\alpha \in \mS} S_\alpha , \cup_{\alpha \in \mS} T_\alpha \right)$. Finally, define the sets:
	\[
	S_\mS \coloneqq \left\{ x \in X : e_G \left( \{x\}, T_\mS' \right) \geq \frac{k}{2} \right\}, T_\mS \coloneqq \left\{ y \in Y : e_G \left( \{y\}, S_\mS' \right) \geq \frac{k}{2} \right\}.
	\]
	
	Let $(S_1,T_1),\ldots,(S_m,T_m)$ be a list of the cuts $\left\{ \left( S_\mS , T_\mS \right) \right\}_{\mS \subseteq \mA}$. By Observation \ref{obs:atom properties}, $m \leq 2^{\left| \mA \right|} = 2^{O(n/k)}$. It remains to prove that in each of these cuts, every vertex is incident to few (i.e., less than $(1 + \eps) k/2$) cross edges, and that every internal cut $(S,T)$ with $|S|>|T|$ is $\eps k$-close to one of the $(S_i,T_i)$s.
	
	Let $\mS \subseteq \mA$. Since, by Observation \ref{obs:atom properties}, every atom is $(40/\delta)$-internal, the number of cross edges with respect to $\left(S'_\mS,T'_\mS \right)$ is at most $|\mA| \frac{40}{\delta} \frac{4nk}{\log n} = O \left( \frac{n^3}{k \log n} \right)$. Therefore there are at most $\varepsilon k/2$ vertices $x \in V(G)$ s.t.\ $\deg1_{G,S'_\mS,T'_\mS}^{\mathrm{Cr}} (x) > k / 2$. Thus $d \left( \left( S'_\mS, T'_\mS \right) , \left( S_\mS , T_\mS \right) \right) \leq \varepsilon k / 2$. Let $x \in V(G)$. By construction,
	\[
	\deg1_{G,S_\mS,T_\mS}^{\mathrm{Cr}} (x) \leq \frac{k}{2} + d \left( \left( S'_\mS, T'_\mS \right) , \left( S_\mS , T_\mS \right) \right) \leq \left( 1 + \eps \right) \frac{k}{2},
	\]
	as desired.
	
	For the second property, let $(S,T)$ be an internal cut. Let $\mS = \{ \alpha_1,\ldots,\alpha_m\}$ be the set of atoms below $\left[ (S,T) \right]$. We will show that $(S,T)$ is equivalent to $(S_\mS, T_\mS)$. By the triangle inequality it suffices to show that $\left( S_\mS',T_\mS' \right)$ is equivalent to $(S,T)$.
	
	Suppose, for a contradiction, that $\left( S_\mS',T_\mS' \right)$ is not equivalent to $(S,T)$. Define
	\[
	[(S',T')] = [(S,T)] \wedge \left[ \left( S_\mS',T_\mS' \right) \right]^c = \left[ (S,T) \cap \left( \bigcap_{\alpha \in \mS} \left( S_\alpha, T_\alpha \right)^c \right) \right].
	\]
	We observe that $[(S,T)]$ is internal and above $[(S',T')]$. Additionally, $[(S',T')]$ is nontrivial. Indeed, by assumption, $(S,T)$ is not equivalent to $(S_\mS',T_\mS')$. Furthermore, by construction, $[(S,T)]$ is above $[(S_\mS',T_\mS')]$. These facts together imply that $[(S',T')]$ is nontrivial. Hence, by Claim \ref{clm:atom internal intersection}, there exists an atom $\alpha$ that is below $[(S,T)]$ (i.e., $\alpha \in \mS$) and has nontrivial meet with $[(S',T')]$. However, by construction, $[(S',T')]$ has trivial meet with all atoms in $\mS$. Thus $\alpha \notin \mS$, a contradiction.
\end{proof}

\section{Properties of Random Subgraphs}\label{sec:pseudorandom}

Let $k = \delta n$, with $\delta = \omega( \log^{-1/3} n)$, and fix a $k$-regular bipartite graph $G = (X \dcup Y,E)$ on $2n$ vertices.
In this section we collect properties of random subgraphs of $G$ that are essential for our proof.

Set
\begin{align*}
&p_1 = \frac{\log n - \log \log \log \log n}{k},\\
&p_2 = \frac{\log n + \log \log \log \log n}{k}.
\end{align*}
We define the following random subgraphs of $G$:
\begin{align*}
&G_2 \sim G(p_2),\\
&G_1 \sim G_2 \left( \frac{p_1}{p_2} \right).
\end{align*}
Observe that $G_1 \sim G(p_1)$. Furthermore, the same distribution on $(G_1,G_2)$ can be obtained as follows. Let $G_1 \sim G(p_1)$, $G' \sim G \lt(\frac{p_2 - p_1}{1 - p_1} \rt)$, and set $G_2 = G_1 \cup G'$.

We will show presently that a.a.s.\ $G_1$ contains isolated vertices,  while $G_2$ does not. Furthermore, the distance between any two vertices that are isolated in $G_1$ is at least $2$ in $G_2$. This motivates the following construction: let $G_1 \subseteq G_H \subseteq G_2$ be the random graph obtained by adding, for each isolated vertex $v$ in $G_1$, an edge drawn uniformly at random from $\{e \in E(G_2) : v \in e \}$. If any of these sets are empty, or if there are two isolated vertices in $G_1$ that are connected in $G_2$, set $G_H = G_1$. The next claim, a variation of \cite[Lemma 7.9]{bollobas1985random}, establishes that it is sufficient to prove that a.a.s.\ $G_H$ contains a perfect matching.

\begin{claim}\label{clm:blue-green}
	Let $Q$ be a monotone increasing property of subgraphs of $G$. If $Q$ holds a.a.s.\ for $G_H$ then, in almost every graph process in $G$, $Q$ holds for $G_{\tau_I}$, the first graph in which there are no isolated vertices.
\end{claim}

We defer the proof until after establishing some properties of $G_1$ and $G_2$.

\begin{claim}\label{lem:isolated}
	A.a.s.\ $G_1$ contains isolated vertices and $G_2$ does not. Furthermore, a.a.s.\ there is no pair $x,y$ of vertices that are isolated in $G_1$ and $xy \in E(G_2)$.
\end{claim}

\begin{proof}
	The probability that a specific vertex is isolated in $G(p)$ is $(1-p)^k$. The expected number of isolated vertices in $G_2$ is therefore:
	\[
	2n(1-p_2)^k \leq 2n \exp \left( - p_2 k \right) = 2n \frac{1}{\omega (n)} = \oone.
	\]
	Applying Markov's inequality, a.a.s.\ $G_2$ contains no isolated vertices.
	
	By a similar calculation, the probability that a specific vertex is isolated in $G_1$ is $\omega (1/n)$. Let the random variable $I$ be the number of vertices in $X$ that are isolated in $G_1$. Then $\E \left[ I \right] = \omegaone$. Furthermore, the events that two vertices $x,y \in X$ are each isolated in $G_1$ are independent. Thus $\var \left[ I \right] \leq \E \lt[ I \rt]$, and by Chebychev's inequality, a.a.s.\ $I>0$ and $G_1$ contains isolated vertices.
	
	For the second part of the claim, observe that by the calculations above a.a.s.\ the number of isolated vertices in $G_1$ is $O(\log n)$. Therefore the expected number of edges between these vertices in $G_2$ is $\oone$, and so by Markov's inequality a.a.s.\ there are none.
\end{proof}

\begin{proof}[Proof of Claim \ref{clm:blue-green}]
We describe a coupling that relates $G_1, G_2$, and $G_H$ to $G_{\tau_I}$. Consider the following random process. For each edge $e$ of $G$, choose a real number $\alpha_e \sim U[0,1]$ uniformly at random from the interval $[0,1]$, all choices independent. Let $G'_1$ and $G'_2$ be the subgraphs of $G$ whose edges are $E(G'_i) = \{e \in E(G):\alpha_e \leq p_i\}$ for $i \in \{1,2\}$. Let $G'_H$ be the random graph obtained by adding, for each isolated vertex in $G'_1$, the edge incident to it in $G'_2$ whose $\alpha$ value is minimal. If $G'_2$ contains isolated vertices or an edge between two vertices that are isolated in $G'_1$, set instead $G'_H = G'_1$.

Observe that the distributions of $(G_1,G_2,G_H)$ and $(G'_1,G'_2,G'_H)$ are identical. Furthermore, the distribution of the random graph process is identical to that of the process in which edges are revealed in increasing order of $\alpha$. With respect to this process, a.a.s. $G'_H$ is a subgraph of $G_{\tau_I}$. Therefore, if $Q$ holds a.a.s.\ for $G_H$, then $Q$ holds a.a.s.\ for $G'_H$, and as $Q$ is monotone increasing, a.a.s.\ $G_{\tau_I} \in Q$.
\end{proof}

For a cut $(S,T)$ we define the set
\[
\Gamma \left( S,T \right) = \left\{ x \in V : \deg1_{G,S,T}^{\mathrm{Cr}} (x) \geq n^{-1/20} k \right\}
\]
of vertices with high cross degree. We remind the reader that a sequence of events occurs with very high probability (w.v.h.p.) if the probabilities of their non-occurence decay at a polynomial rate.

\begin{lem}
	Let $(S,T)$ be a cut. Suppose $R \subseteq V \setminus \Gamma (S,T)$ is a set of $O(\log n)$ isolated vertices in $G_1$. Then w.v.h.p.\ for every $x \in R$, $\deg1_{G_H}^{\mathrm{Cr}}(x) = 0$.
\end{lem}

\begin{proof}
	It suffices to show that a.a.s.\  for every $x \in R$, $\deg1_{G_2}^{\mathrm{Cr}}(x) = 0$. Indeed, the expected number of cross edges incident to $x$ in $G_2$ is bounded above by $n^{-1/20}k \frac{p_2 - p_1}{1-p_1} = \tilde{O} \lt(n^{-1/20} \rt)$. The conclusion follows by applying Markov's inequality and a union bound over the $O(\log n)$ vertices.
\end{proof}

\begin{lem}\label{lem:low cross degree}
	Let $(S,T)$ be a cut, and let $V' = V \setminus \Gamma (S,T)$. Then, w.v.h.p.\ for every $x \in V'$, $\deg1_{G_1}^{\mathrm{Cr}} (x) \leq 30$.
\end{lem}

\begin{proof}
	Observe that $\deg1_{G_1}^{\mathrm{Cr}} (x) \sim Bin \left( \deg1_{G}^{\mathrm{Cr}} (x) , p_1 \right)$. Therefore:
	\[
	\Prob \left[ \deg1_{G_1}^{\mathrm{Cr}} (x) \geq 30 \right] \leq \binom{n^{-1/20} k}{30} p_1^{30} = \tilde{O} \left( \frac{1}{n^{3/2}} \right).
	\]
	The lemma follows by applying a union bound over all $O(n)$ vertices in $V'$.
\end{proof}

\begin{lem}\label{lem:large_dist}
	Let $(S,T)$ be a cut, and let $V_{\mathrm{low}}$ be the set of vertices $x \in V \setminus \Gamma (S,T)$ such that $\deg1_{G_1}^{\mathrm{Par}} (x) \leq \ldc \log n$. W.v.h.p.\ the following hold:
	\begin{enumerate}
		\item $|V_{\mathrm{low}}|\leq n^{0.01}$.
		
		\item\label{itm:dist} For each $x,y \in V_{\mathrm{low}}$, the distance between $x$ and $y$ in $G_H$ is at least $6$.
	\end{enumerate}
\end{lem}

\begin{proof}
	We first show that w.v.h.p.\ $|V_{\mathrm{low}}|\leq n^{0.01}$. Indeed, suppose $x \in V$ satisfies $\deg1_G^{\mathrm{Cr}} (x) < n^{-1/20}k$. The probability that $x \in V_{\mathrm{low}}$ is at most 
	\begin{align*}
	    \sum_{i=0}^{\ldc \log n} \Pr \left[ \deg1_{G_1}^{\mathrm{Par}} (x) = i\right]
		\leq & \frac{\log n}{1000} \binom{k}{\ldc \log n} p_1^{\ldc \log n}  (1-p_1)^{\lt(1-O(n^{-1/20}) \rt)k}\\
		\leq & \frac{\log n}{1000} \left( \frac{e \cdot k p_1 }{\ldc \log n} \right)^{\ldc \log n}  \exp\left(- \lt(1-O(n^{-1/20}) \rt)k p_1 \right) \\
		\leq & \left(1000 e \right)^{\ldc \log n} \cdot \tilde{O}\left(\frac{1}{n}\right) < n^{-0.991}.
	\end{align*}
	Thus, $\E \left[ |V_{\mathrm{low}}| \right] < n^{0.009}$. Therefore, by Markov's inequality, $\Prob \left[ |V_{\mathrm{low}}| \geq n^{0.01} \right] \leq n^{-0.001}$.
	
	The proof of \ref{itm:dist} is similar to the proof of \cite[Claim 4.4]{ben2011resilience} and property (P2) in \cite[Lemma 5.1.1]{RomanThesis}. Fix two distinct vertices $u,w\in V \setminus \Gamma (S,T)$ and consider a path $(u=v_0,\ldots,v_r=w)$  in $G$, where $1\leq r\leq 5$. Denote by $\mA$ the event that for every $0 \leq i \leq r-1$, we have $\{v_i, v_{i+1}\}\in E(G_2)$, i.e., the path exists in $G_2$. Denote by $\mB$ the event that $u,w \in V_{\mathrm{low}}$. Clearly, $\Prob\lt[\mA  \rt]=p_2^r$, hence
	\begin{equation*}
	\Prob \lt[\mB\wedge\mA\rt] = p_2^r \cdot \Prob \lt[\mB|\mA\rt].
	\end{equation*}
	Let $X$ denote the random variable which counts the number of parallel edges in $G_2$ incident with $u$ or $w$ disregarding the pairs $\{u,v_1\}$, $\{v_{r-1},w\}$, and $\{u,w\}$. Observing that  $X\sim \text{Bin}\lt( (1 - \oone)2k , p_2 \rt)$ and using standard concentration inequalities, we have
	\[
	\Prob \lt[\mB|\mA\rt] \leq\Prob\lt[X<2 \ldc \log n \rt] < n^{-1.8}.
	\]
	Fixing the two endpoints $u,w$, the number of such sequences is at most $k^{r-1}$. Applying a union bound over all pairs of vertices and possible paths between them, we conclude that the probability of a path in $G_2$ of length $r \leq 5$ connecting two distinct vertices of $V_{\mathrm{low}}$ is at most
	\[
	\sum_{r=1}^{5}n^2\cdot k^{r-1}\cdot p_2^r\cdot n^{-1.8} = \tilde{O} \left( \frac{1}{n^{0.8}} \right).
	\]
	This completes the proof of the lemma.
\end{proof}

Recall that a set $A \subset V$ is \termdefine{partite} if $A \subset X$ or $A \subset Y$. A vertex is a \termdefine{parallel neighbor} of $A$ if it is connected to a vertex in $A$ via a parallel edge.

\begin{lem}\label{lem:g1sse}
	Let $(S,T)$ be a cut. W.v.h.p.\ the following holds. If $A \subseteq V$ is a partite set satisfying
    \begin{itemize}
    \item $|A| \leq n^{0.9}$, and
    \item for every $x \in A$, $\deg1_{G_1}^{\mathrm{Par}} (x) \geq \ldc \log n$,
 	\end{itemize}
  then $A$ has at least $|A| \econst \log n$ parallel neighbors in $G_1$.
\end{lem}

\begin{proof}
	Let $A \subseteq V$ be a partite set, and let $t = t(A) = |A| \econst\log n$. Let $\mathcal{P}(A)$ be the event that the minimum parallel degree of a vertex in $A$ is at least $\ldc \log n$.
	
	For any fixed set $B$,
	\begin{align*}
		\Prob \left[ N_{G_1}(A) \subseteq B \land \mathcal{P}(A) \right] \leq \Prob \left[ e_{G_1} (A,B) \geq 2 t \right]
		\leq \binom{e_G(A,B)}{2 t}p_1^{2t}\\
		\leq \binom{|A||B|}{2t} p_1^{2t} \leq \left( \frac{e \cdot |A||B|p_1}{2t} \right)^{2t}.
	\end{align*}
	Applying a union bound, we have:
	\begin{align*}
		& \Prob \left[ \exists A \text{ s.t.\ } |A| \leq n^{0.9} \land \mathcal{P}(A) \land \left| N_{G_1}(A) \right| \leq t(A) \right]\\
		& \leq 2 \sum_{|A|=1}^{n^{0.9}} {\binom{n}{|A|}\binom{n}{t(A)} \left( \frac{e |A| t(A) p_1}{2 t(A)} \right)^{2t(A)}}
		\leq  2 \sum_{|A|=1}^{n^{0.9}} \left( \frac{n e}{|A|} \right)^{|A|} \left(\frac{n e^3 |A|^2 p_1^2}{4t(A)}\right)^{t(A)}\\
		& \leq 2 \sum_{|A|=1}^{n^{0.9}} \left( \frac{n e}{|A|} \right)^{|A|} \left(\frac{e^3   |A|^2 \log^2(n)}{4 \delta^2 t(A) n}\right)^{t(A)}
		\leq \sum_{|A|=1}^{n^{0.9}} n^{-t(A)/20} = O \left( \frac{1}{n^{1/20}} \right). \qedhere
	\end{align*}
\end{proof}

\section{Proof of \thmref{thm:main}}\label{sec:proof}

\subsection{Outline}

As mentioned previously, we prove Theorem \ref{thm:main} by showing that a.a.s.\ $G_H$ does not contain a Hall cut. This is similar to the approach used in \cite{erdos1964random} to show that $p = \log n / n$ is the threshold for $K_{n,n} (p)$ to contain a perfect matching. There, a union bound over all cuts $(S,T)$ (satisfying certain conditions) was sufficient for the result. In this regard, the crucial property of $K_{n,n}$ is that every cut has many outgoing edges. Essentially the same approach was utilized in the proof of \cite[Lemma 3.1]{luria2019threshold} to show that if a $k$-regular, bipartite graph $G$ satisfies a certain expansion property, then the threshold for $G(p)$ to contain a perfect matching is $p = \Theta \left(\log n / k\right)$. However, for arbitrary $G$, there may be many cuts $(S,T)$ with few outgoing edges, potentially foiling the union bound. Indeed, this is the case in the counterexamples described in Appendix \ref{app:proof_counterexample}.

To overcome this we take a more delicate approach, wherein we group the various cuts in $G$ into families that we treat separately. Informally, the steps are as follows:

\begin{enumerate}
	\item The first family contains all cuts that are not internal (in the sense of Definition \ref{def:internal}), and therefore have many outgoing edges in $G$. Here a simple union bound suffices to show that a.a.s.\ none of these cuts are Hall cuts in $G_H$ (Claim \ref{clm:bigexpansion}).

	\item At this point we apply Lemma \ref{lem:structure} to conclude that any cut not covered in the previous step is close to one of the $m = 2^{\Theta(n/k)}$ cuts from the lemma. We fix one of these cuts, $(S',T')$, and show that conditioned on $G_H$ having no isolated vertices, w.v.h.p.\ none of the cuts that are $\varepsilon k$-close to $(S',T')$ becomes a Hall cut in $G_H$. As $m$ is subpolynomial, a union bound implies that a.a.s.\ $G_H$ does not contain a Hall cut.
	
	For a cut $(S,T)$, we define the set of \termdefine{shifted} vertices:
	\[
	\Delta = \Delta (S,T) = \left( S \setminus S' \right) \cup \left( S' \setminus S \right) \cup \left( T \setminus T' \right) \cup \left( T' \setminus T \right).
	\]
	We make use of the natural correspondence $\Delta \leftrightarrow (S,T)$, and interchange between them freely. We recall the definition of the set
    \[
    \Gamma = \Gamma (S',T') = \left\{ x \in V : \deg1_{G,S',T'}^{\mathrm{Cr}} (x) \geq n^{-1/20} k \right\}
    \]
    of vertices with many cross edges in $G$ w.r.t.\ $(S',T')$. We emphasize that $\Gamma$ is deterministic, i.e., depends only $G$ and $(S',T')$.
	
	\item The second family of cuts consists of those with $|\Delta| \geq n^{0.9}$. The insight here is that shifting a large number of vertices w.r.t.\ $(S',T')$ creates many cross edges. Here too, a union bound suffices to show that w.v.h.p.\ none of these are Hall cuts in $G_H$ (Claim \ref{clm:deltaupper}).
	
	\item The third family consists of cuts satisfying $\left| \Delta \right| \leq \frac{n^{-1/20}}{10}	\left| \Gamma \right|$. As the vertices in $\Gamma$ have, by definition, many cross edges, if $\Delta$ is much smaller than $\Gamma$ then most of these cross edges are unaffected. This allows us to employ a union bound here as well (Claim \ref{clm:gammaupper}).

	\item We now argue that w.v.h.p.\ we may remove from $G_H$ a small matching $M$ covering all vertices that have low degree in $G_1$, leaving the residual graph ${\widetilde{G}_H = G_H \setminus V(M)}$. In Observation \ref{obs:miracle}, we show that if $G_H$ contains a Hall cut that is $C$-close to $(S',T')$, then $\widetilde{G}_H$ contains a Hall cut that is $C$-close to
	\[({S' \setminus V(M) ,} {T' \setminus V(M)}).\]
	 It therefore suffices to consider cuts in $\widetilde{G}_H$.
	
	\item\label{itm:x4} It remains to consider $\Delta$ such that $\frac{n^{-1/20}}{10} \left| \Gamma \right| < |\Delta| < n^{0.9}$. We show that w.v.h.p.\ there is no such Hall cut if either $|\Delta| \geq \log n / \log\log n$ or $\Delta \cap \Gamma \neq \emptyset$ (Claim \ref{clm:win_by_gamma}). Here we take advantage of the fact that once the low degree vertices have been removed from $G_1$, the remaining vertices satisfy an expansion property (Claim \ref{clm:sse}).
	
	\item\label{itm:x5} Finally, we argue that w.v.h.p.\ there is no Hall cut satisfying $|\Delta| \leq \log n / \log\log n$ and $\Delta \cap \Gamma = \emptyset$ (Claim \ref{clm:x5}). Here we use the expansion property to show that such a cut cannot exist: w.v.h.p.\ $(S',T')$ contains many outgoing edges in $G_1$, and it is impossible to make all these edges parallel by shifting only $\log n/ \log \log n$ vertices.
	
\end{enumerate}

\subsection{The proof}

We first show that if a cut has many outgoing edges, the probability that it is a Hall cut in $G_H$ is very small.

\begin{claim}\label{clm:bigexpansion}
	A.a.s.\ $G_H$ contains no Hall cut $(S,T)$ that is not internal.
\end{claim}

\begin{proof}
	Since $G_1 \subseteq G_H$ it suffices to prove the statement with $G_H$ replaced by $G_1$. Suppose $(S,T)$ is not internal. Then it has at least $4nk/\left( \log n \right)$ cross edges. By Observation \ref{obs:probHall} the probability that it is a Hall cut is less than
	\[
	(1-p_1)^{e_G(S,T^c)} \leq \exp \left( - p_1 \frac{2nk}{\log n} \right) = \exp( - (1 - \oone) 2n ) = o \left( 4^{-n} \right).
	\]
	The claim follows by applying a union bound over all $4^n$ cuts in $G$.
\end{proof}

We now apply Lemma \ref{lem:structure} to obtain the cuts $(S_1,T_1)\ldots,(S_m,T_m)$. By Claim \ref{clm:bigexpansion}, Lemma \ref{lem:structure}, and the fact that $m$ is subpolynomial, it suffices to show that w.v.h.p.\ all cuts $(S,T)$ that are $\varepsilon k$-close to $(S_i,T_i)$ for some $i$ are not Hall cuts in $G_H$.

Fix an index $i \in [m]$, set $S'=S_i,T'=T_i$, and define $\Gamma$ and $\Delta$ with respect to $(S',T')$ as in the outline. Henceforth, cross edges, parallel edges, cross degrees and parallel degrees are with respect to $(S',T')$.

\begin{claim}\label{clm:deltaupper}
	W.v.h.p.\ for every $\Delta$ such that $n^{0.9} \leq \left|\Delta\right| \leq \eps k$, $(S,T)$ is not a Hall cut in $G_H$.
\end{claim}

\begin{proof}
	By Lemma \ref{lem:structure}, each $x \in \Delta$ satisfies $\deg1_G^{\mathrm{Par}}(x) \geq (1-\eps)\frac{k}{2}$. Most of these parallel edges - all those with an endpoint not in $\Delta$ - are cross edges w.r.t.\ $(S,T)$. Thus the number of cross edges satisfies:
	\[
	e_G(S,T^c) + e_G(S^c,T) \geq (1-\eps) |\Delta| \frac{k}{2} - |\Delta|^2 \geq \frac{|\Delta| k}{3}.
	\]

	By Observation \ref{obs:probHall} the probability that $(S,T)$ is a Hall cut in $G_1$ is at most:
	\[
	(1-p_1)^{|\Delta|k/6} \leq \left( \frac{1}{n} \right)^{|\Delta|/7}.
	\]
	Applying a union bound, the probability that there exists such a Hall cut is at most
	\[
	\sum_{|\Delta| = n^{0.9}}^{\eps n} \binom{2n}{|\Delta|} \left( \frac{1}{n} \right)^{|\Delta|/7} = O \left( \frac{1}{n} \right).
	\qedhere
	\]
\end{proof}

\begin{claim}\label{clm:gammaupper}
	W.v.h.p.\ for all $\Delta$ s.t.\ $\left| \Delta \right| \leq \frac{n^{-1/20}}{10} \left|\Gamma\right|$, the corresponding cut $(S,T)$ is not a Hall cut in $G_H$.
\end{claim}

\begin{proof}
	Suppose $\Delta$ satisfies the claim's hypothesis. By Lemma \ref{lem:structure} and the definition of $\Gamma$, each $x \in \Gamma$ satisfies 
\[
\min \left\{ \deg1_G^{\mathrm{Cr}} \left( x \right), \deg1_G^{\mathrm{Par}} \left( x \right) \right\} \geq n^{-1/20} k.
\]
Ignoring, for the moment, the possibility that $N_G(x) \cap \Delta \neq \emptyset$, this means that every $x \in \Gamma$ is incident to at least $n^{-1/20} k$ cross edges w.r.t.\ $(S,T)$, regardless of whether $x \in \Delta$. There are at most $|\Delta| \min \{|\Gamma|, k\}$ edges between $\Delta$ and $\Gamma$. Accounting for possible double counting of the edges incident to $\Gamma$, we obtain:
	\[
	e_G(S,T^c) + e_G(S^c,T) \geq \frac{|\Gamma| n^{-1/20} k}{2} - |\Delta| \min \{|\Gamma|, k\} \geq 4 |\Delta| k.
	\]
	Applying Observation \ref{obs:probHall}, the probability that $(S,T)$ is a Hall cut in $G_1$ is at most
	\[
	(1-p_1)^{4 |\Delta|k/2} \leq \left(\frac{1}{n}\right)^{1.9|\Delta|}.
	\]
	
	We now observe that if $\Delta = \emptyset$ (i.e., $(S,T) = (S',T')$), then the probability that $(S,T)$ is a Hall cut in $G_1$ is at most $(1-p_1)^k = \tilde{O} (1/n)$. Let $X$ be the number of cuts satisfying the claim's hypothesis that are Hall cuts in $G_1$. Applying a union bound, we have:
	\[
	\Prob \left[ X > 0 \right] \leq \tilde{O} \left( \frac{1}{n} \right) + \sum_{1 \leq |\Delta| \leq \min \left\{ \frac{n^{-1/20}}{10}\left| \Gamma \right|, n^{0.9} \right\}} \binom{2n}{|\Delta|} \left( \frac{1}{n} \right)^{1.9|\Delta|} = O \left( \frac{1}{n^{0.9}} \right).
	\qedhere
	\]
\end{proof}

\begin{rmk}\label{rmk:smallgamma}
	As a consequence of Claim \ref{clm:gammaupper}, if $\left| \Gamma \right| \geq 10 n^{19/20}$ then w.v.h.p.\ none of the cuts that are $\eps k$-close to $(S',T')$ become Hall cuts in $G_H$. This is because Claim \ref{clm:deltaupper} covers all cases where $\left| \Delta \right| \geq n^{0.9}$, and the previous claim covers all cases where $\left| \Delta \right| \leq \frac{n^{-1/20}}{10} \left|\Gamma\right|$. Therefore, we proceed under the assumption that $\left| \Gamma \right| < 10 n^{19/20}$.
\end{rmk}

Before continuing to steps \ref{itm:x4} and \ref{itm:x5}, we modify $G_H$ by removing a small matching covering the low degree vertices that are not in $\Gamma$. Moreover, this matching contains only parallel edges. The following claim, together with Lemma \ref{lem:large_dist}, implies that conditioned on $G_H$ having no isolated vertices, w.v.h.p.\ such a matching exists.

\begin{clm}\label{clm:few_cross}
	Conditioned on there being no isolated vertices in $G_H$, w.v.h.p.\ every vertex in $V \setminus \Gamma$ is incident to at least one parallel edge in $G_H$.
\end{clm}

\begin{proof}
	Assuming there are no isolated vertices in $G_H$, if there exists some $v \in V\setminus \Gamma$ s.t.\ $\deg1_{G_H}^{\mathrm{Par}} (v) = 0$, then $\deg1_{G_1}^{\mathrm{Par}} (v) = 0$ and $\deg1_{G_2}^{\mathrm{Cr}} (v) > 0$. We use the first moment method to show that w.v.h.p.\ there are no vertices $v \notin \Gamma$ for which this holds. Indeed, if $v \notin \Gamma$ then the probability of this occurring is bounded from above by
	\[
	\frac{k}{n^{1/20}}p_2 (1-p_1)^{k\left( 1 - n^{-1/20} \right)}
	= \tilde{O} \left(\frac{1}{n^{21/20}}\right).
	\]
	Therefore the expected number of such vertices is $\tilde{O} \left( n^{-1/20} \right)$. By Markov's inequality, w.v.h.p.\ there are none.
\end{proof}

Recall that
\[
V_{\mathrm{low}} = \left\{ x \in V \setminus \Gamma : \deg1_{G_1}^{\mathrm{Par}} (x) \leq\ldc \log n \right\}.
\]
Conditioning on the conclusions of Lemma \ref{lem:large_dist} and Claim \ref{clm:few_cross} holding, there exists a matching $M \subseteq G_H$ of size $\left|V_{\mathrm{low}}\right|$ consisting of parallel edges that contains $V_{\mathrm{low}}$.

\begin{claim}\label{clm:N(Vlow)}
	W.v.h.p.\ $N_{G_H}(V_{\mathrm{low}}) \cap \Gamma = \emptyset$.
\end{claim}

\begin{proof}
	By Remark \ref{rmk:smallgamma} we may assume $\left| \Gamma \right| < n^{0.96}$. Fix an arbitrary vertex $x\not\in \Gamma$. Then
	\begin{align*}
		\Prob \left[ x \in V_{\mathrm{low}} \land N_{G_H}(x) \cap \Gamma\neq \emptyset \right] &\leq 
        \sum_{y \in \Gamma}{\Prob \left[ x \in V_{\mathrm{low}} \land y \in N_{G_H}(x)\right] }\\
        &\leq \sum_{y \in \Gamma}{\Prob \left[ \lt| N^{\mathrm{Par}}_{G_1}(x)\setminus\{y\} \rt| \leq \ldc \log n \land y \in N_{G_2}(x)\right] }\\
        &=\sum_{y \in \Gamma}{\Prob \left[\lt| N^{\mathrm{Par}}_{G_1}(x)\setminus\{y\} \rt| \leq \ldc \log n \right] \cdot \Prob \left[y \in N_{G_2} (x) \right]}\\
		&\leq |\Gamma| \left(\frac{1}{n}\right)^{0.99} \cdot p_2 = O \left( \frac{1}{n^{1.01}} \right),
	\end{align*}
	where the probability of the first event is estimated as in the proof of \lemref{lem:large_dist}. The equality between the second and third lines is due to the fact that the events $\lt| N^{\mathrm{Par}}_{G_1}(x)\setminus\{y\} \rt| \leq \ldc \log n $ and $y \in N_{G_2} (x)$ are independent. The statement of the claim follows from a union bound over all $O(n)$ choices of $x$.
\end{proof}

\begin{claim}\label{clm:crossremoved}
	W.v.h.p.\ the number of cross edges incident to $V(M)$ in $G$ is $o \left( n^{0.99} \right)$.
\end{claim}

\begin{proof}
	By \lemref{lem:large_dist} and \claref{clm:N(Vlow)}, we may assume that $|M| = |V_{\mathrm{low}}|\leq n^{0.01}$ and $V(M) \cap \Gamma = \emptyset$. Therefore, each vertex in $V(M)$ is incident to $O \left(n^{0.95}\right)$ cross edges, and the claim follows.
\end{proof}

Observe that the identity of $M$ depends only on the parallel edges of $G_H$ w.r.t.\ $(S',T')$. This allows us to think of $G_1$ as being exposed in two independent stages. In the first stage the parallel edges of $G_1$ are exposed, and in the second the cross edges are exposed.

Set
\begin{align*}
&\tg = G \setminus V(M),\\
&\tg_1 = G_1 \setminus V(M),\\
&\tg_H = G_H \setminus V(M),\\
&\left( \tS , \tT \right) = \left( S' \setminus V(M), T' \setminus V(M) \right),
\end{align*}
and
\[\Gamma_{\mathrm{bad}} = \left\{ x \in \Gamma : \deg1_{G_1}^{\mathrm{Par}} (x) < \ldc \log n \right\}.\]
For a cut $(S,T)$ we define the set of shifted vertices with respect to $(\tS,\tT)$ as
\[
\tDelta = \tDelta(S \setminus V(M), T \setminus V(M)) = \Delta(S,T) \setminus V(M).
\]

\begin{observation}\label{obs:miracle}
	Suppose that $G_H$ has a Hall cut $(S,T)$ whose shifted vertex set is $\Delta$. Then $\tDelta$ corresponds to a Hall cut in $\tg_H$, and $|\Delta| - |V(M)| \leq |\tDelta| \leq |\Delta|$.
\end{observation}

\begin{proof}
	Observe that there is no edge connecting $S$ and $T^{c}$. Therefore, $|S \cap V(M)| \leq |T\cap V(M)|$,
	and so $(S \setminus V(M), T \setminus V(M))$ is also a Hall cut in $\tg_H$.
	
	 Since $\tDelta$ is obtained from $\Delta$ by removing at most $|V(M)|$ vertices, the conclusion follows.
\end{proof}

By Observation \ref{obs:miracle}, it suffices to show that w.v.h.p.\ there are no Hall cuts in $\tg_H$ with $\frac{n^{-1/20}}{10} |\Gamma| - |V(M)| \leq |\tDelta| \leq n^{0.9}$.

\begin{clm}\label{clm:sse}
	W.v.h.p\ every partite set $A \subseteq V \setminus \left( V(M) \cup \Gamma_{\mathrm{bad}} \right)$ of size at most $n^{0.9}$ satisfies
	\[
	\left| N^{\mathrm{Par}}_{\tg_1} (A) \right| \geq |A| \egconst \log n.
	\]
\end{clm}

\begin{proof}
	Suppose the conclusion does not hold, i.e., there is a partite set $A \subseteq V(\tg_1) \setminus \Gamma_{\mathrm{bad}}$ with $|A| \leq n^{0.9}$ s.t.\ $\left| N_{\tg_1} (A) \right| < |A| \egconst \log n$. Then, for every $x \in A$, $\deg1_{G_H}^{\mathrm{Par}} (x) \geq \ldc \log n$. Furthermore,
	\[
	N^{\mathrm{Par}}_{G_1}(A) \subseteq N^{\mathrm{Par}}_{\tg_1}(A) \cup \left( \left( V_{\mathrm{low}} \cup N^{\mathrm{Par}}_{G_H} (V_{\mathrm{low}}) \right) \cap N^{\mathrm{Par}}_{G_1}(A)\right).
	\]
	However, $\left| \left( V_{\mathrm{low}} \cup N^{\mathrm{Par}}_{G_H} (V_{\mathrm{low}}) \right) \cap N^{\mathrm{Par}}_{G_1}(A) \right| \leq |A|$, because if a vertex in $A$ has two neighbors in $V_{\mathrm{low}} \cup N^{\mathrm{Par}}_{G_H} (V_{\mathrm{low}})$, then there are two vertices in $V_{\mathrm{low}}$ whose distance in $G_H$ is at most $4$, contradicting Lemma \ref{lem:large_dist}. Therefore:
	\[
	\left| N^{\mathrm{Par}}_{G_1}(A) \right| \leq \left| N^{\mathrm{Par}}_{\tg_1} (A) \right| + |A| \leq |A| \egconst \log n + |A| < |A| \econst \log n.
	\]
	The set $A$ does not satisfy the conclusion of Lemma \ref{lem:g1sse}, which holds w.v.h.p. Therefore the conclusion of the present claim holds w.v.h.p.\ as well.
\end{proof}

\begin{clm}
	W.v.h.p.\ $\left| \Gamma_{\mathrm{bad}} \right| \leq |\Gamma| / n^{0.4}$.
\end{clm}

\begin{proof}
	By Lemma \ref{lem:structure}, every vertex has parallel degree in $G$ at least $(1-\eps)\frac{k}{2}$. Therefore the probability that a vertex's parallel degree in $G_1$ is less than $\ldc \log n$ is bounded above by
	\[
	\ldc \log n \binom{(1-\eps)\frac{k}{2}}{\ldc \log n} p_1^{\ldc \log n} (1-p_1)^{(1-\eps)k/2 - \ldc \log n} = O \left( n^{-0.49} \right).
	\]
	The conclusion follows from an application of Markov's inequality.
\end{proof}

Henceforth, unless otherwise specified, parallel degrees, cross degrees, etc., are with respect to the vertex set $V \setminus V(M)$ and the cut $\left( \tS,\tT \right)$.

\begin{clm}\label{clm:win_by_gamma}
	The following holds w.v.h.p. Suppose $\tDelta$ of size $\frac{n^{-1/20}}{10} |\Gamma| - |V(M)| \leq \left| \tDelta \right| \leq n^{0.9}$ satisfies one of:
    \begin{enumerate}
        \item\label{itm:large delta} $|\tDelta| \geq \log n / \log\log n$.
        \item\label{itm:small delta} $\tDelta \cap \Gamma \neq \emptyset$.
    \end{enumerate}
    Then $\tDelta$ does not correspond to a Hall cut in $\tGH$.
\end{clm}

\begin{proof}
	Let $(S,T)$ be the cut corresponding to $\tDelta$. Set:
	\[
	a = \tS \setminus S, b = \tS^c \setminus S^c, c = \tT^c \setminus T^c, d = \tT \setminus T.
	\]
	
	We note that $|\Gamma_{\mathrm{bad}}| = o \left( |\tDelta| / \log n \right)$. Indeed, by the previous claim $|\Gamma_{\mathrm{bad}}| \leq |\Gamma| / n^{0.4}$. Now, if $|V(M)| \leq \frac{n^{-1/20}}{20} |\Gamma|$ then $|\tDelta| = \Omega (n^{-1/20} |\Gamma|)$ and the conclusion follows. Otherwise, $|\Gamma| = O \left( n^{1/20} |V(M)| \right)$. By Lemma \ref{lem:large_dist} $|V(M)| \leq 2|V_{\mathrm{low}}| = O \left( n^{0.01} \right)$. Therefore
	\[
	|\Gamma_{\mathrm{bad}}| \leq \frac{|\Gamma|}{n^{0.4}} \leq \frac{n^{1/20} |V(M)|}{n^{0.4}} = O \left( n^{-0.34} \right) = \oone.
	\]
	Thus, in fact, $\Gamma_{\mathrm{bad}} = \emptyset$, and so $|\Gamma_{\mathrm{bad}}| = o \left( |\tDelta| / \log n \right)$.
	
	We also observe that these calculations imply
	\begin{equation}\label{eq:gamma over delta}
	\frac{|\Gamma|}{|\tDelta|} = O \left( n^{0.06} \right).
	\end{equation}
	
	We now observe that if either $N_{\tg_1}^{\mathrm{Par}} (b) \nsubseteq c$ or $N_{\tg_1}^{\mathrm{Par}} (d) \nsubseteq a$, then $(S,T)$ is not a Hall cut. Thus we may assume that $N_{\tg_1}^{\mathrm{Par}} (b) \subseteq c$ and $N_{\tg_1}^{\mathrm{Par}} (d) \subseteq a$. The conclusion of Claim \ref{clm:sse} then implies that
	\[
	|a| \geq |d \setminus \Gamma_{\mathrm{bad}}| \egconst \log n, |c| \geq |b \setminus \Gamma_{\mathrm{bad}}| \egconst \log n.
	\]
	Since $\left| \Gamma_{\mathrm{bad}} \right| = o \left( \left| \tDelta \right| / \log n \right)$:
	\begin{align*}
	|b| + |d| & = |b \setminus \Gamma_{\mathrm{bad}}| + |b \cap \Gamma_{\mathrm{bad}}| + |d \setminus \Gamma_{\mathrm{bad}}| + |d \cap \Gamma_{\mathrm{bad}}|\\
	& \leq O \left( \frac{|a| + |c|}{\log n} \right) + |\Gamma_{\mathrm{bad}}| = O \left( \frac{|\tDelta|}{\log n} \right).
	\end{align*}
	We may assume that $|S|>|T|$, for otherwise $(S,T)$ is not a Hall cut by definition. It also holds that:
	\begin{align*}
	|S'| - |T'| & = |\tS| - |\tT| = \left( |S| + |a| - |b| \right) - \left( |T| + |d| - |c| \right)\\
	& > |a| + |c| - \left( |b| + |d| \right) = \left( 1 - O \left( \frac{1}{\log n} \right) \right)|\tDelta|.
	\end{align*}
	Since $G$ is $k$-regular,
	\[
	e_G(\tS,\tT^c) \geq \left( 1 - O \left( \frac{1}{\log n} \right) \right)|\tDelta| k.
	\]
   
    By Claim \ref{clm:crossremoved} the number of cross edges in $G$ that are not cross edges in $\tg$ is at most $n^{0.99}$. Thus:
    \[
	e_{\tilde{G}}(\tS,\tT^c) \geq \left( 1 - O \left( \frac{1}{\log n} \right) \right)|\tDelta| k.
	\]
    Set $\Delta_1 = \tDelta \cap \Gamma, \Delta_2 = \tDelta \setminus \Gamma$. We then have:
    \begin{align*}
    \left| E_{\tg} (S,T^c) \cap E_{\tg} (\tS, \tT^c) \right| & \geq e_{\tilde{G}} (\tS,\tT^c) - |\Delta_1|(1+\eps)\frac{k}{2} - |\Delta_2|\frac{k}{n^{1/20}}\\
    & \geq \left( \frac{1}{2} - \eps \right)|\Delta_1|k + \left( 1 - O \left( \frac{1}{\log n} \right) \right) |\Delta_2|k.
    \end{align*}
    Therefore, the probability that none of the cross edges is in $\tgone$ is at most:
    \[
    (1 - p_1)^{\left( \frac{1}{2} - \eps \right)|\Delta_1|k} (1-p_1)^{\left( 1 - O \left( \frac{1}{\log n} \right) \right) |\Delta_2|k}.
    \]
    Suppose \ref{itm:large delta} holds. Let $m = \max \left\{ \frac{n^{-1/20}}{10} |\Gamma| - |V(M)|, \log n / \log\log n \right\}$. Then, applying a union bound over all choices of $\Delta_1 \subseteq \Gamma$ and $\Delta_2$:
    \begin{align*}
    \alpha & \coloneqq 
    \sum_{\substack{|\tDelta| \in \{m, \ldots, n^{0.9}\} \\|\Delta_1|+|\Delta_2| = |\tDelta|}}\binom{|\Gamma|}{|\Delta_1|} \binom{2n}{|\Delta_2|}(1 - p_1)^{\left( \frac{1}{2} - \eps \right)|\Delta_1|k} (1-p_1)^{\left( 1 - O \left( \frac{1}{\log n} \right) \right) |\Delta_2|k}\\
    & \leq \sum_{\substack{|\tDelta| \in \{m, \ldots, n^{0.9}\} \\|\Delta_1|+|\Delta_2| = |\tDelta|}}|\Delta_1|^{-|\Delta_1|} |\Delta_2|^{-|\Delta_2|} \left( e |\Gamma| \left( \frac{1}{n} \right)^{1/2 - 2 \eps} \right)^{|\Delta_1|} \left( O( \log \log \log n) \right)^{|\Delta_2|}.
    \end{align*}
    Now
    \[
    |\Delta_1|^{-|\Delta_1|} |\Delta_2|^{-|\Delta_2|} \leq \left( \frac{2}{|\tDelta|} \right)^{|\tDelta|} = \left( \frac{2}{|\tDelta|} \right)^{|\Delta_1|} \left( \frac{2}{|\tDelta|} \right)^{|\Delta_2|}.
    \]
    Thus:
    \begin{align*}
    \alpha & \leq \sum_{\substack{|\tDelta| \in \{m, \ldots, n^{0.9}\} \\|\Delta_1|+|\Delta_2| = |\tDelta|}} \left( \frac{2 e|\Gamma|}{|\tDelta|} \left( \frac{1}{n} \right)^{1/2 - 2 \eps} \right)^{|\Delta_1|} \left( O \left( \frac{\log \log \log n}{|\tDelta|} \right) \right)^{|\Delta_2|}\\
    & \stackrel{\eqref{eq:gamma over delta}}{\leq} \sum_{\substack{|\tDelta| \in \{m, \ldots, n^{0.9}\} \\|\Delta_1|+|\Delta_2| = |\tDelta|}} \left( \frac{1}{n^{2/5}} \right)^{|\Delta_1|} \left( O \lt(\frac{\log|\tDelta|}{|\tDelta|}\rt) \right)^{|\Delta_2|}
    \leq n^{0.9} \left( \tilde{O} \left( \frac{1}{\log n} \right) \right)^{\log n / \log\log n}\\
    & \leq n^{0.9} \frac{1}{n^{1- \oone}} = O \left( \frac{1}{n^{0.05}} \right).
    \end{align*}
    Otherwise, \ref{itm:small delta} holds. Then $\left| \Delta_1 \right| \geq 1$. By a similar application of a union bound, the probability that there exists any Hall cut $(S,T)$ in $\tgone$ satisfying the hypothesis is bounded above by:
    \[
   \sum_{\substack{|\tDelta| \in [\log n / \log \log n] \\|\Delta_1|+|\Delta_2| = |\tDelta|}} \left( \frac{1}{n^{2/5}} \right)^{\left|\Delta_1\right|} \left( O \lt(\frac{\log|\tDelta|}{|\tDelta|}\rt) \right)^{|\Delta_2|}
    = \tilde{O} \left( \frac{1}{n^{2/5}} \right).
    \qedhere
    \]
\end{proof}

It remains to show that w.v.h.p.\ there is no Hall cut with $\left| \tDelta \right| \leq \log n / \log\log n$ and $\tDelta \cap \Gamma = \emptyset$.

\begin{claim}\label{clm:x5}
	W.v.h.p.\ there exists no Hall cut $(S,T)$ in $\tGH$ with $|\tDelta|< \log n / \log\log n$ and $\tDelta \cap \Gamma = \emptyset$.
\end{claim}

\begin{proof}
	We first show that if $\left| \tS \right| \leq \left| \tT \right|$ then w.v.h.p.\ there is no Hall cut satisfying the claim's hypothesis. Suppose that $\tDelta$ corresponds to a Hall cut $(S,T)$ and $\tDelta \cap \Gamma = \emptyset$. We will show that $|\tDelta| = \Omega (\log n)$.
	
	Recall the definition of $a,b,c,$ and $d$ from the previous proof. It holds that:
	\begin{align*}
	\left|S\right| - \left|T\right| = \left|\tS\right| - \left|\tT\right| - |a| + |b| - |c| + |d| \implies\\
	|b|+|d| = |S|-|T| + |\tT| - |\tS| + |a| + |c| > 0.
	\end{align*}	
	Since $\tDelta \cap \Gamma = \emptyset$, we have $b,d \subseteq V \setminus \left( V(M) \cup \Gamma \right)$. Since $(S,T)$ is a Hall cut, for every $x \in b \cup d$, $N_{\tg_1}(x) \subseteq \Delta$. However, by Claim \ref{clm:sse}, w.v.h.p.\ for every such $x$, $\left| N_{\tg_1}(x) \right| = \Omega \left( \log n \right)$. Therefore $|\tDelta| = \Omega (\log n)$, as claimed.
	
	We now assume that $\left| \tS \right| - \left| \tT \right| > 0$.	We will show presently that w.v.h.p.\ $e_{\tg_1} \left( \tS , \tT^c \right) = \Omega \left( \log n \right)$. Suppose $(S,T)$ is a cut satisfying the claim's hypothesis. Then $\tDelta$ must contain a vertex cover of $E_{\tg_1} \left( \tS , \tT^c \right)$. However, since $\tDelta \cap \Gamma = \emptyset$, by Lemma \ref{lem:low cross degree} w.v.h.p.\ each vertex in $\tDelta$ is incident to at most $30$ cross edges in $G_1$. Since $\left| \tDelta \right| < \log n / \log\log n$, $\tDelta$ does not contain a vertex cover of $E_{\tg_1} \left( \tS , \tT^c \right)$, and so $(S,T)$ is not a Hall cut.
	
	Finally, we show that w.v.h.p.\ $e_{\tg_1} \left( \tS , \tT^c \right) = \Omega \left( \log n \right)$. Since $G$ is $k$-regular, we have $e_G \left( \tS , \tT \right) \geq k \left( \left| \tS \right| - \left| \tT \right| \right) \geq k$. By Claim \ref{clm:crossremoved}, the number of cross edges of $G$ incident to $V(M)$ is $o \left( n^{0.99} \right)$, so $\tg$ has at least $C = (1-o(1))k$ cross edges. Now, $e_{\tgone} \left( \tS , \tT \right) \sim Bin(C, p_1)$. By an application of Chernoff's inequality, w.v.h.p.\ $e_{\tgone} \left( \tS , \tT \right) \geq \frac{1}{2} C p_1 = \Omega \left( \log n \right)$.
\end{proof}

\appendix

\section{Proof of Proposition \ref{prop:counterexample}}\label{app:proof_counterexample}

It may be intuitive to think at first - as all three of us did - that the conclusion of \thmref{thm:main} holds for all large regular bipartite graphs, i.e., the requirement $k = \omega \left( \frac{n}{\log^{1/3} n} \right)$ is not necessary. In this section, we analyze a construction of Goel, Kapralov, and Khanna \cite{goel2010perfect} to show that this is not true, and indeed for small values of $k$, $G(p)$ might not contain a perfect matching even for relatively large $p$.

The intuition for all of our counterexamples comes from the following simple construction. 

\begin{definition}
	A \termdefine{$k$-resistor} between two vertices $x$ and $y$ is the following bipartite graph: The vertex set is $\{x,y\} \dcup X' \dcup Y'$, where $X'$ and $Y'$ have cardinality $k$. Let $x' \in X', y' \in Y'$ be ``special'' vertices. The edge set is:
	\[
	\left\{ xx', yy' \right\} \cup \left(\left\{ ab : a \in X', b \in Y' \right\} \setminus \left\{ x'y' \right\}\right).
	\]
	In other words, starting from the complete bipartite graph on $X'$ and $Y'$, the edge $x'y'$ is removed, and the edges $xx'$ and $yy'$ are added.
\end{definition}

Notice that of the $2k+2$ vertices of a $k$-resistor between $x$ and $y$, all but $x$ and $y$ have degree $k$. Furthermore, if a spanning subgraph of the resistor contains a perfect matching, both edges $xx'$ and $yy'$ are present. This leads to the following construction. 

\begin{prop}\label{prop:1resistor}
	Construct a $k$-regular, $n = (2k^2+2)$-vertex bipartite graph $G$ as follows. Let $x$ and $y$ be two initial vertices. Add $k$ distinct $k$-resistors between $x$ and $y$. Then, a.a.s.\  the random subgraph $G(p)$ does not contain a perfect matching for any $p = o \left( n^{-1/4} \right)$. On the other hand, a.a.s.\ $G(p)$ contains no isolated vertices for any $p = \omega \left( \log n / \sqrt{n} \right)$.
\end{prop}

\begin{proof}
	Both conclusions follow from the first moment method.
	
	Let $H \sim G(p)$. Note that $H$ contains a perfect matching only if for one of the resistors, both edges $xx'$ and $yy'$ are present. This occurs with probability $p^2$. As there are $k$ different resistors, and they are all edge-disjoint, the expected number of such pairs is $kp^2$. Since $k = \Theta(\sqrt{n})$, if $p = o \left( n^{-1/4} \right)$, a.a.s.\ there is no such pair in $H$.
	
	The expected number of isolated vertices in $H$ is $n(1-p)^k \leq  \exp \left( \log n - pk \right)$. When $p = \omega \left( \log n  / \sqrt{n} \right)$ this tends to zero, and a.a.s.\ there are no isolated vertices.
\end{proof}

In this example we had $k=\Theta(\sqrt{n})$, leaving a large gap between it and the range $k=\Theta(n)$ in \thmref{thm:main}. We reduce this gap as follows.

\begin{definition}\label{def:series}
	A \termdefine{$(k,\ell,r)$-series of resistors} between two vertices $x$ and $y$ is constructed as follows. Let $K_1,K_2,\ldots,K_\ell$ be $\ell$ copies of the complete bipartite graph $K_{k,k}$, with respective vertex sets $X_1 \dcup Y_1, X_2\dcup Y_2,\ldots,X_\ell\dcup Y_\ell$. For each $1\leq i \leq \ell$, let $x_i^1,x_i^2,\ldots,x_i^r \in X_i, y_i^1,y_i^2,\ldots,y_i^r \in Y_i$ be distinct. Remove all edges of the form $x_i^jy_i^j$, and add all edges of the form $y_i^j x_{i+1}^j$, as well as $xx_1^j,y_\ell^jy$.
\end{definition}

The following proposition uses a construction similar to the one in Proposition \ref{prop:1resistor}.

\begin{prop}\label{prop:ellresistor}
	For $n= 2+ 20 k \log k \log \log k$, construct a $k$-regular $n$-vertex bipartite graph $G$ as follows. Starting with two vertices $x$ and $y$, add $\log k$ distinct  $\left( k, 10 \log \log k, \frac{k}{\log k} \right)$-series of resistors between $x$ and $y$. A.a.s.\ the random subgraph $G(p)$ does not contain a perfect matching for any $p \leq 2 \log n / k$. On the other hand, $p = \left( \log n + \omegaone \right) / k$ suffices for $G(p)$ to contain no isolated vertices a.a.s.
\end{prop}

\begin{proof}
	For consistency with Definition \ref{def:series}, let $\ell = 10 \log \log k$ and $r = k/\log k$. For a spanning subgraph $G'\subseteq G$ to contain a perfect matching, there must be at least one series of resistors containing at least one edge of the form $x x_1^j$, at least one edge of the form $y_\ell^j y$, and one edge of the form $y_i^j x_{i+1}^j$ for each $i$ between $1$ and $\ell -1$. Therefore, applying the union bound over all $k/r$ choices of the $(k,\ell,r)$-series, we obtain
	\begin{align*}
	\Prob \lt[ G(p) \mbox{ contains a perfect matching} \rt] \leq 
	\frac{k}{r}\lt[1-(1-p)^r\rt]^{\ell +1}.
	\end{align*}
	Let $p = 2 \log n / k$. Then $(1-p)^r \sim e^{-2}$, and therefore
	\begin{align*}
	\prob \lt[ G(p) \mbox{ contains a perfect matching} \rt] \leq \log k \lt( 1 - e^{-2} \rt)^{10 \log \log k} = \oone.
	\end{align*}
	
	The statement about isolated vertices follows from an argument similar to the one in the proof of Proposition \ref{prop:1resistor}.
\end{proof}

\bibliography{pm_in_gp}

\providecommand{\bysame}{\leavevmode\hbox to3em{\hrulefill}\thinspace}
\providecommand{\MR}{\relax\ifhmode\unskip\space\fi MR }
\providecommand{\MRhref}[2]{%
  \href{http://www.ams.org/mathscinet-getitem?mr=#1}{#2}
}
\providecommand{\href}[2]{#2}
\begin{thebibliography}{10}

\bibitem{ajtai1981longest}
Mikl{\'o}s Ajtai, J{\'a}nos Koml{\'o}s, and Endre Szemer{\'e}di, \emph{The
  longest path in a random graph}, Combinatorica \textbf{1} (1981), no.~1,
  1--12.

\bibitem{ben2011resilience}
Sonny Ben-Shimon, Michael Krivelevich, and Benny Sudakov, \emph{On the
  resilience of {H}amiltonicity and optimal packing of {H}amilton cycles in
  random graphs}, SIAM J. Discrete Math. \textbf{25} (2011), no.~3, 1176--1193.

\bibitem{bollobas1984evolution}
B{\'e}la Bollob{\'a}s, \emph{The evolution of sparse graphs}, Graph Theory and
  Combinatorics ({C}ambridge 1983).

\bibitem{bollobas1985random}
\bysame, \emph{Random graphs}, Academic Press, 1985.

\bibitem{bregman1973some}
Lev Bregman, \emph{Some properties of nonnegative matrices and their
  permanents}, Soviet Math. Dokl, vol.~14, 1973, pp.~945--949.

\bibitem{egorychev1981solution}
Gregory Egorychev, \emph{The solution of van der {W}aerden's problem for
  permanents}, Advances in Mathematics \textbf{42} (1981), no.~3, 299--305.

\bibitem{erdds1959random}
Paul Erd\H{o}s and Alfr{\'e}d R{\'e}nyi, \emph{On random graphs {I}}, Publ.
  Math. Debrecen \textbf{6} (1959), 290--297.

\bibitem{erds1960evolution}
\bysame, \emph{On the evolution of random graphs}, Publ. Math. Inst. Hung.
  Acad. Sci \textbf{5} (1960), 17--61.

\bibitem{erdos1964random}
\bysame, \emph{On random matrices}, Magyar Tud. Akad. Mat. Kutat{\'o} Int.
  K{\"o}zl \textbf{8} (1964), 455--461.

\bibitem{falikman1981proof}
Dmitry Falikman, \emph{Proof of the van der {W}aerden conjecture regarding the
  permanent of a doubly stochastic matrix}, Mathematical notes of the Academy
  of Sciences of the USSR \textbf{29} (1981), no.~6, 475--479.

\bibitem{RomanThesis}
Roman Glebov, \emph{On {H}amilton cycles and other spanning structures}, Ph.D.
  thesis, Free University of Berlin, 2013.

\bibitem{glebov2017threshold}
Roman Glebov, Humberto Naves, and Benny Sudakov, \emph{The threshold
  probability for long cycles}, Combinatorics, Probability and Computing
  \textbf{26} (2017), no.~2, 208--247.

\bibitem{goel2010perfect}
Ashish Goel, Michael Kapralov, and Sanjeev Khanna, \emph{Perfect matchings via
  uniform sampling in regular bipartite graphs}, ACM Transactions on Algorithms
  (TALG) \textbf{6} (2010), no.~2, 27.

\bibitem{komlos1973hamilton}
J{\'a}nos Koml{\'o}s and Endre Szemer{\'e}di, \emph{Hamilton cycles in random
  graphs}, Infinite and finite sets \textbf{2} (1973), 1003--1010.

\bibitem{komlos1983limit}
\bysame, \emph{Limit distribution for the existence of {H}amiltonian cycles in
  a random graph}, Discrete Math. \textbf{43} (1983), no.~1, 55--63.

\bibitem{korshunov1976solution}
Alexey Korshunov, \emph{Solution of a problem of {E}rd{\H{o}}s and {R}\'{e}nyi
  on {H}amilton cycles in non-oriented graphs}, Soviet Math. Dokl., vol.~17,
  1976, pp.~760--764.

\bibitem{krivelevich2014robust}
Michael Krivelevich, Choongbum Lee, and Benny Sudakov, \emph{Robust
  {H}amiltonicity of {D}irac graphs}, Trans. Amer. Math. Soc. \textbf{366}
  (2014), no.~6, 3095--3130.

\bibitem{krivelevich2015long}
\bysame, \emph{Long paths and cycles in random subgraphs of graphs with large
  minimum degree}, Random Structures \& Algorithms \textbf{46} (2015), no.~2,
  320--345.

\bibitem{kuhn2014robust}
Daniela K{\"u}hn, Allan Lo, Deryk Osthus, and Katherine Staden, \emph{The
  robust component structure of dense regular graphs and applications},
  Proceedings of the London Mathematical Society \textbf{110} (2014), no.~1,
  19--56.

\bibitem{luria2019threshold}
Zur Luria and Michael Simkin, \emph{On the threshold problem for {Latin}
  boxes}, Random Structures \& Algorithms \textbf{55} (2019), no.~4, 926--949.

\bibitem{posa1976hamiltonian}
Lajos P{\'o}sa, \emph{Hamiltonian circuits in random graphs}, Discrete Math.
  \textbf{14} (1976), no.~4, 359--364.

\bibitem{riordan2014long}
Oliver Riordan, \emph{Long cycles in random subgraphs of graphs with large
  minimum degree}, Random Structures \& Algorithms \textbf{45} (2014), no.~4,
  764--767.

\bibitem{sudakov2017robustness}
Benny Sudakov, \emph{Robustness of graph properties}, Surveys in Combinatorics
  2017 \textbf{440} (2017), 372.

\end{thebibliography}
\bibliographystyle{amsplain}

\end{document}